\documentclass[11pt,twoside]{article}

\usepackage{mathrsfs,amsfonts,amsmath,amssymb}
\usepackage{latexsym}
\newtheorem{satz}{Theorem}
\newtheorem{proposition}[satz]{Proposition}
\newtheorem{theorem}[satz]{Theorem}
\newtheorem{lemma}[satz]{Lemma}

\newtheorem{corollary}[satz]{Corollary}
\newtheorem{remark}[satz]{Remark}

\def\eps{\varepsilon}
\def\_phi{\varphi}

\def\F{{\mathbb F}}

\def\C{{\mathbb C}}
\def\R{{\mathbb R}}
\def\E{\mathsf {E}}
\def\T{{\mathbb T}}

\def\Z_N{{\mathbb Z}_N}
\def\Z{{\mathbb Z}}

\def\f{{\mathbb F}}

\def\D{{\mathbb D}}

\def\G{\Gamma}

\def\D{\Delta}

\def\T{\mathsf {T}}

\author{Shkredov I.D.}
\title{ On tripling constant of multiplicative subgroups
\footnote{
This work was supported by grant
Russian Scientific Foundation RSF 14--11--00433.}
}
\date{}
\begin{document}
\maketitle

\begin{center}
 Annotation.
\end{center}

{\it \small
    We prove that any multiplicative subgroup $\G$ of the prime field $\f_p$
    with $|\G| < \sqrt{p}$ satisfies $|3\G| \gg \frac{|\G|^2}{\log |\G|}$.
    Also, we obtain a bound for the multiplicative energy of any nonzero shift
    of $\G$, namely $\E^{\times} (\G+x) \ll |\G|^{2} \log |\G|$, where $x\neq 0$ is an arbitrary.
}
\\

\section{Introduction}
\label{sec:introduction}

Let $p$ be a prime number, $\F_p$ be the finite field, and $\F^*_p = \F_p \setminus \{0\}$.
Let also $\G \subseteq \F^*_p$ be an arbitrary  multiplicative subgroup.
Such subgroups were studied by various authors (see the references in \cite{KS1}).
One of the interesting question is the determination of the {\it additive} structure of multiplicative subgroups,
see e.g. \cite{Waring_Z_p,Waring_Z_p_new,Glibichuk_zam,Hart_A+A_subgroups,ss,Sh_ineq,Sh_average,sv}.
In particular, what can we say about the size of sumsets of subgroups, that is
about the sets of the form
$$
    2\G = \G + \G := \{ \gamma_1+\gamma_2 ~:~ \gamma_1, \gamma_2 \in \G \}
    \,?
$$

There is a well--known  conjecture that the sumset $2\G$ contains $\F^*_p$, provided that $|\G| > p^{1/2+\eps}$,
where $\eps>0$ is any number and $p\ge p(\eps)$ is large enough.
In the article we study a bigger set $3\G = \G+\G+\G$ instead of $2\G$.
Let us formulate the main result of the paper.

\begin{theorem}
    Let $p$ be a prime number, $\G \subset \f_p^*$ be a multiplicative subgroup, $|\G| < \sqrt{p}$.
    Then
$$
    |3\G| \gg \frac{|\G|^2}{\log |\G|} \,.
$$
\label{t:intr_1}
\end{theorem}

It is interesting to compare Theorem \ref{t:intr_1} with a
result of
A.A. Glibichuk who
obtained
in \cite{Glibichuk_zam} that $|4\G| > p/2$ provided $|\G| > \sqrt{p}$
as well as with a result from \cite{Sh_average} :
$\f_p^* \subseteq 5\G$ if $-1\in \G$ and $|\G| \gg \sqrt{p} \log^{1/3} p$.

Let us say a few words about the proof.
In \cite{R_Minkovski} O. Roche--Newton
obtained
that for any set $A$ from $\R$ there are $a,b\in A$ such that
\begin{equation}\label{f:pinned}
    |(A+a) (A+b)| \gg \frac{|A|^2}{\log |A|} \,.
\end{equation}
More precisely, it was proved in \cite{R_Minkovski} that the common  multiplicative energy
(see the definition in the next section \ref{sec:definitions}) of $A+a$ and $A+b$ is small
\begin{equation}\label{f:pinned'}
    \E^{\times} (A+a,A+b) \ll |A|^{2} \log |A| \,.
\end{equation}
The proof used the Szemer\'{e}di--Trotter theorem from the incidence geometry.
Roche--Newton calculated the number of collinear triples in the Cartesian product $A \times A$
in two different ways and comparing the estimates gives (\ref{f:pinned'}).
In our arguments we use Stepanov's method \cite{Stepanov} in form of Mit'kin~\cite{Mitkin}
(see also \cite{K_Tula} and \cite{KS1}) which allows us to get (\ref{f:pinned}), (\ref{f:pinned'})
for $A$
be
any multiplicative subgroup of size less than $\sqrt{p}$.
It is easy to see that such an analog of (\ref{f:pinned}) implies Theorem \ref{t:intr_1}.
Note also that in the case of multiplicative subgroup $A$ bound (\ref{f:pinned'}) is equivalent to
$$
    \E^{\times} (A+1) \ll |A|^{2} \log |A|
$$
because of $A+a = a(A+1)$, $A+b = b(A+1)$, $a,b\in A$.
Thus, the method allows us to obtain a good upper  bound for the multiplicative energy of $A+1$
(and actually of any shift $A+x$, $x\in \F^*_p$ see Theorem \ref{t:E_GPi} of section \ref{sec:proof}).


\section{Notation}
\label{sec:definitions}


Let $f,g : \f_p \to \C$ be two functions.
Put
\begin{equation}\label{f:convolutions}
    (f*g) (x) := \sum_{y\in \f_p} f(y) g(x-y) \quad \mbox{ and } \quad
        (f\circ g) (x) := \sum_{y\in \f_p} f(y) g(y+x)
\end{equation}
Replacing $+$ by the multiplication, one can define the {\it multiplicative convolution}
of two functions $f$ and $g$.
Write $\E^{+}(A,B)$ for the {\it additive energy} of two sets $A,B \subseteq \f_p$
(see e.g. \cite{TV}), that is
$$
    \E^{+} (A,B) = |\{ a_1+b_1 = a_2+b_2 ~:~ a_1,a_2 \in A,\, b_1,b_2 \in B \}| \,.
$$
If $A=B$ we simply write $\E^{+} (A)$ instead of $\E^{+} (A,A).$
Clearly,
\begin{equation}\label{f:energy_convolution}
    \E^{+} (A,B) = \sum_x (A*B) (x)^2 = \sum_x (A \circ B) (x)^2 = \sum_x (A \circ A) (x) (B \circ B) (x)
    \,.
\end{equation}
By $|S|$ denote the cardinality of a set $S \subseteq \f_p$.
Note
that
$$
    \E^{+} (A,B) \le \min \{ |A|^2 |B|, |B|^2 |A|, |A|^{3/2} |B|^{3/2} \} \,.
$$
In the same way define the {\it multiplicative energy} of two sets $A,B \subseteq \f_p$
$$
    \E^{\times} (A,B) = |\{ a_1 b_1 = a_2 b_2 ~:~ a_1,a_2 \in A,\, b_1,b_2 \in B \}| \,.
$$
Certainly, multiplicative energy $\E^{\times} (A,B)$ can be expressed in terms of multiplicative convolutions,
similar to (\ref{f:energy_convolution}).

Let $\G \subseteq \f_p^*$ be a multiplicative subgroup.
A set $Q\subseteq \f_p^*$ is called {\it $\G$--invariant} if $Q\G = Q$.
All logarithms are base $2.$ Signs $\ll$ and $\gg$ are the usual Vinogradov's symbols.

\section{On sumsets of multiplicative subgroups}
\label{sec:ms}

In the section we have deal with the quantity ($\T$ for {\it collinear triples})
$$
    \T (A,B,C,D) := \sum_{c \in C,\, d\in D} \E^\times (A-c,B-b) \,.
$$
Because $\E^\times (A-c,B-b) \ge |A| |B|$ it follows that $\T (A,B,C,D) \ge |A||B||C||D|$.
It turns out that there is the same upper bound for $\T$ up to logarithmic factors
in the case of $A,B,C,D$ equal some cosets of a multiplicative subgroup.
The proof based on the following lemma of Mit'kin~\cite{Mitkin}, see also \cite{SSV}.


\begin{lemma}
\label{l:mitkin}
Let $p>2$ be a prime number, $\Gamma,\Pi$ be subgroups of $\mathbb{F}_p^*,$
$M_\Gamma,M_\Pi$ be sets of distinct coset representatives of $\Gamma$ and $\Pi$, respectively.
For an arbitrary set $\Theta \subset M_\Gamma \times M_\Pi$ such that $(|\Gamma||\Pi|)^2|\Theta| < p^3$
and $|\Theta| \le 33^{-3}|\Gamma||\Pi|$, we have
\begin{equation}\label{eq:mitkin}
\sum_{(u,v) \in \Theta}\Bigl|\{(x,y) \in \Gamma \times \Pi : ux+vy=1\}\Bigr| \ll (|\Gamma||\Pi||\Theta|^2)^{1/3}.
\end{equation}
\end{lemma}

Using the above lemma,
we prove the main technical result of the section.
The proof is in spirit of \cite{R_Minkovski}.

\begin{proposition}
Let $p$ be a prime  number, $\Gamma,\Pi$ be subgroups of $\mathbb{F}_p^*$.
Suppose that $|\Gamma||\Pi| < p$.
Then
\begin{equation}\label{f:sigma}
    \sum_{\gamma \in \G,\, \pi\in \Pi} \E^\times (\G-\gamma,\Pi-\pi)
        \ll
            |\G|^2 |\Pi|^2 \log (\min\{ |\G|, |\Pi| \}) + |\G| |\Pi| (|\G|^2 + |\Pi|^2) \,.
\end{equation}
\label{p:sigma}
\end{proposition}
\begin{proof}
Consider the equation
\begin{equation}\label{tmp:31.03.2015_1}
    (a-b)(a'-c') = (a-c)(a'-b') \,, \quad a,b,c\in \G \,,\quad a',b',c' \in \Pi \,.
\end{equation}
Clearly, the number of its solutions is
$$
    \T (\G,\Pi,\G,\Pi) = \sum_{\gamma \in \G,\, \pi\in \Pi} \E^\times (\G-\gamma,\Pi-\pi) \,.
$$
One can assume that the products in (\ref{tmp:31.03.2015_1}) are nonzero and $b\neq c$
because otherwise we have at most $O( |\G|^3 |\Pi| + |\G| |\Pi|^3 + |\Pi|^2 |\G|^2 )$ number of the solutions.
Denote by $\sigma$ the remaining number of the solutions.

Take a parameter $\tau \ge 2$ and put
$$
    \Theta_\tau := \{ (u,v) \in M_\G \times M_\Pi ~:~ | \{ (x,y) \in \Gamma \times \Pi : ux+vy=1 \} | \ge \tau \} \,.
$$
In other words, $\Theta_\tau$ counts the number of lines $l_{u,v} = \{ (x,y) : ux+vy=1 \}$,
$(u,v) \in M_\G \times M_\Pi$ having
the intersection with $\Gamma \times \Pi$ greater than $\tau$.
Obviously, if $(u,v) \equiv (u',v') {\rm ~mod~} (\G \times \Pi)$ then the intersections of lines
$l_{u,v}$ and $l_{u',v'}$ with $\Gamma \times \Pi$  are coincide.
By Lemma \ref{l:mitkin}, we have $|\Theta_\tau| \ll |\G| |\Pi| \tau^{-3}$ provided
$(|\Gamma||\Pi|)^2|\Theta_\tau| < p^3$ and $|\Theta_\tau| \le 33^{-3} |\Gamma||\Pi|$.
Thus
\begin{equation}\label{f:q_tau}
    q_\tau :=
    \{ (u,v) ~:~ | \{ (x,y) \in \Gamma \times \Pi : ux+vy=1 \} | \ge \tau \} \ll |\G|^2 |\Pi|^2 \tau^{-3}
\end{equation}
provided
$(|\Gamma||\Pi|)^2|\Theta_\tau| < p^3$ and $|\Theta_\tau| \le 33^{-3} |\Gamma||\Pi|$.
The number of all lines intersecting $\G\times \Pi$ by at least two points does not exceed $|\G|^2 |\Pi|^2$.
Thus, splitting $\Theta_\tau$ onto smaller sets if its required, we get  upper bound (\ref{f:q_tau})
for
$q_\tau$
with possibly bigger absolute constant,
provided the only condition  $(|\Gamma||\Pi|)^2|\Theta_\tau| < p^3$ holds.
The
assumption
$|\Gamma||\Pi| < p$ implies the last inequality.

It is easy to see that for any tuple $(a,a',b,b',c,c')$ satisfying (\ref{tmp:31.03.2015_1}),
the
points $(a,a')$, $(b,b')$, $(c,c')$ lies at the same line and these points are pairwise distinct.
Clearly, the number of such triples belonging the lines having the form $ux+vy=0$ and intersecting
$\Gamma \times \Pi$ does not exceed
$(|\G||\Pi|)^2$,
so it is negligible.
Thus, using (\ref{f:q_tau}), we see that the rest of the quantity $\sigma$ is less than
$$
    \sum_{u,v} |l_{u,v} \cap (\Gamma \times \Pi)|^3
        \ll
            \sum_{j\ge 1}\, \sum_{u,v ~:~ 2^{j-1} < |l_{u,v} \cap (\Gamma \times \Pi)| \le 2^{j}}
                |l_{u,v} \cap (\Gamma \times \Pi)|^3
                    \ll
$$
$$
                    \ll
                                \sum_{j\ge 1}\,  2^{3j} \cdot |\G|^2 |\Pi|^2 2^{-3j}
                                    \ll
                                         |\G|^2 |\Pi|^2 \log (\min\{ |\G|, |\Pi| \}) \,.
$$
This completes the proof.
$\hfill\Box$
\end{proof}

\bigskip

\begin{remark}
    Careful analysis of the proof gives that one can assume that $a,b,c$ belong to different cosets of $\G$
    and $a',b',c'$  are from different cosets of $\Pi$
    (it will be three Cartesian products of cosets instead of one in the case).
    In particular, the following holds
\begin{equation}\label{f:sigma'}
    \sum_{\gamma \in \xi \G,\, \pi\in \eta \Pi} \E^\times (\G-\gamma,\Pi-\pi)
        \ll
            |\G|^2 |\Pi|^2 \log (\min\{ |\G|, |\Pi| \}) + |\G| |\Pi| (|\G|^2 + |\Pi|^2) \,,
\end{equation}
    where $\xi,\eta\in \f_p^*$ are arbitrary.
    Of course, one can
    permute
    $\G$ to $\xi \G$ and $\Pi$ to $\eta \Pi$ in formula (\ref{f:sigma'}).
\label{r:plus_cosets}
\end{remark}

Proposition \ref{p:sigma} allows us to prove new results on sumsets of subgroups,
which improve some bounds from \cite{Waring_Z_p_new}, see Lemma 7.3
and also
Lemma 7.4.

\begin{corollary}
    Let $p$ be a prime number, $\G \subset \f_p^*$ be a multiplicative subgroup, $|\G| < \sqrt{p}$.
    Then
$$
    \left| \left\{ \frac{a \pm b}{a \pm c} ~:~ a,b,c\in \G \right\} \right| \gg \frac{|\G|^2}{\log |\G|} \,,
$$
    and for any $X\subseteq \G$ one has
$$
    |2\G + X| \gg \frac{|X|^2}{\log |\G|} \,.
$$
    In particular
$$
    |3\G| \gg \frac{|\G|^2}{\log |\G|} \,.
$$
\end{corollary}
\begin{proof}
    The first estimate follows from the Cauchy--Schwarz inequality and the  interpretation
    of the quantity $\T (\G,\Pi,\G,\Pi)$ as the number of the solutions of (\ref{tmp:31.03.2015_1}).
    To get the second estimate  apply (\ref{f:sigma'}) with $\G=\G$, $\Pi = \G$, $\xi=\eta=-1$.
    We find $\gamma_1, \gamma_2 \in \G$ such that
$$
    \E^{\times} (\G+\gamma_1, \G+\gamma_2) \ll |\G|^2 \log |\G| \,.
$$
    By the Cauchy--Schwarz inequality, we get
$$
    |(\G+\gamma_1) (X+\gamma_2)| \cdot  \E^{\times} (\G+\gamma_1, \G+\gamma_2)
        \ge
            |(\G+\gamma_1) (X+\gamma_2)| \cdot  \E^{\times} (\G+\gamma_1, X+\gamma_2)
                \ge
                    |\G|^2 |X|^2 \,.
$$
    Note that $(\G+\gamma_1) (X+\gamma_2) \subseteq 2\G + \gamma_1 X + \gamma_1 \gamma_2$.
    Moreover, $|2\G + \gamma_1 X + \gamma_1 \gamma_2| = |2\G + X|$.
    Hence
$$
    |2\G + X| \ge |(\G+\gamma_1) (X+\gamma_2)| \gg \frac{|X|^2}{\log |\G|}
$$
    as required.
    $\hfill\Box$
\end{proof}

\bigskip

We are going to apply the method  of the section to the problems concerning decompositions
of multiplicative subgroups in the future paper.

\section{Generalizations}
\label{sec:proof}

First of all,
we derive a consequence of Proposition \ref{p:sigma} concerning multiplicative energies of shifts of subgroups.

\begin{theorem}
    Let $p$ be a prime  number, $\Gamma,\Pi$ be  multiplicative subgroups of $\mathbb{F}_p^*$.
    Suppose that $|\Gamma||\Pi| < p$.
    Then for any $x,y \neq 0$ one has
$$
    \E^\times (\Gamma + x, \Pi+y) \ll |\G| |\Pi| \log (\min\{ |\G|, |\Pi| \}) + |\G|^2 + |\Pi|^2 \,.
$$
\label{t:E_GPi}
\end{theorem}
\begin{proof}
Since $x,y \neq 0$ it follows that $x\in \xi \G$, $y\in \eta \Pi$ and $\xi, \eta \neq 0$.
Further, it is easy to see that
$$
    \E^\times (\Gamma + x, \Pi+y) = \E^\times (\xi^{-1} \Gamma + \gamma, \eta^{-1} \Pi+\pi)
$$
for {\it any} $\gamma \in \G$ and $\pi \in \Pi$.
Thus, all energies in the left--hand side of formula (\ref{f:sigma'}) are coincide.
This completes the proof.
$\hfill\Box$
\end{proof}

\bigskip

It is interesting to compare the last theorem with results of  \cite{U} and \cite{MV} which give a pointwise bound
for the multiplicative convolution of characteristic functions of multiplicative subgroups
in contrary to our average estimate.

\bigskip

Using a formula $$\E^{+} (\G) = \E^{\times} (\G,\G+1)$$ for an arbitrary subgroup $\G$,
we
derive by the Cauchy--Schwarz inequality and Theorem \ref{t:E_GPi}
that $\E^{+} (\G) \ll |\G|^{5/2} \log^{1/2} |\G|$.
This coincides with Konyagin's bound \cite{K_Tula} up to logarithmic factors.

\bigskip

Let us prove a generalization of Proposition \ref{p:sigma} and Theorem \ref{t:E_GPi}.

\begin{theorem}
    Let $p$ be a prime number, $\Gamma,\Pi$ be  multiplicative subgroups of $\mathbb{F}_p^*$.
    Suppose that $|\Gamma||\Pi| < p$ and $Q_1$ is $\G$--invariant, $Q_2$ is $\Pi$--invariant sets.
    Then
\begin{equation}\label{f:E_Q1Q2}
    \T (Q_1,Q_2,Q_1,Q_2)
            \ll
                        (|Q_1| |Q_2|)^3 (|\G| |\Pi|)^{-1} \log^2 (\min\{ |Q_1|, |Q_2|\})
                    +  |Q_1| |Q_2| (|Q_1|^2 + |Q_2|^2) \,.
\end{equation}
\label{t:E_Q1Q2}
\end{theorem}
\begin{proof}
Let $L = \log (\min\{ |Q_1|, |Q_2|\})$.
We use the arguments of Proposition \ref{p:sigma}.
The term $|Q_1| |Q_2| (|Q_1|^2 + |Q_2|^2) + |Q_1|^2 |Q_2|^2$ appears similarly as in the proof and thus we are considering the set of
lines (pairs)
$$
    \mathcal{L}_\tau := \{ (u,v) ~:~ | \{ (x,y) \in Q_1 \times Q_2 : ux+vy=1 \} | \ge \tau \}
$$
intersecting $Q_1\times Q_2$ in at least three distinct points.
Let $Q_1\times Q_2 = \bigsqcup_{i=1}^s C_i$, where $C_i$ are products of the correspondent cosets,
$s=|Q_1| |Q_2| |\G|^{-1} |\Pi|^{-1}$.
Taking a line $l\in \mathcal{L}_\tau$ and using the Dirichlet principle,
we find a number $\D (l)$
such that
$$
 \tau \le |l \cap (Q_1 \times Q_2)| \le \sum_{i=1}^s |l\cap C_i| \ll L \D(l) |\Omega_\D (l)| \,,
$$
where
$$
    \Omega_\D (l) = \{ i ~:~ \D < |l \cap C_i| \le 2 \D \} \,,
$$
and $\D(l) \ge \max\{ \tau s^{-1}, 1\}$.
The number $\D (l)$ depends on $l$ but using the Dirichlet principle again, we find a set
$\mathcal{L}'_\tau \subseteq \mathcal{L}_\tau$, $|\mathcal{L}'_\tau| \gg |\mathcal{L}_\tau| L^{-1}$
with some fixed $\D \ge \max\{ \tau s^{-1}, 1\}$.
After that, using the arguments of Proposition \ref{p:sigma}, we see that
$$
    |\mathcal{L}_\tau| L^{-1} \ll |\mathcal{L}'_\tau| \ll \frac{|\G|^2 |\Pi|^2}{\D^3} \ll \frac{|\G|^2 |\Pi|^2 s^3}{\tau^3}
$$
and we have
obtained (\ref{f:E_Q1Q2}).

Let us give another proof.
Take the same family of the lines $\mathcal{L}'_\tau$ and consider a smaller
family of points $\mathcal{P}' := \bigcup_{l\in \mathcal{L}'_\tau} \bigsqcup_{i\in \Omega_\D (l)} C_i$.
Using Lemma \ref{l:mitkin} as well as the arguments of the proof of Proposition \ref{p:sigma} again, we see that
any line meets at most $|\G| |\Pi| \D^{-3}$ cells $C_i$.
In other words, $|\Omega_\D (l)| \ll |\G| |\Pi| \D^{-3}$.
Let us calculate the number of indices $I(\mathcal{L}'_\tau, \mathcal{P}')$ between lines from $\mathcal{L}'_\tau$
and points $\mathcal{P}'$.
On the one hand, any line from $\mathcal{L}'_\tau$ contains at least $\D |\Omega_\D (l)| \gg \tau L^{-1}$ number of points.
Thus,
$$
    I(\mathcal{L}'_\tau, \mathcal{P}') \gg \D |\mathcal{L}'_\tau| |\Omega_\D (l)| \gg |\mathcal{L}'_\tau| \tau L^{-1} \,.
$$
On the other hand, by a trivial estimate for the number of indices between points and lines
(see e.g. \cite{TV}, section 8.2), we get
$$
    I(\mathcal{L}'_\tau, \mathcal{P}') \le \sum_{i=1}^s I(\mathcal{L}'_\tau, \mathcal{P}' \cap C_i)
        \le
            \sum_{i=1}^s \left( |\mathcal{P}' \cap C_i| |L_i|^{1/2} + |L_i| \right) \,,
$$
where by $L_i$ we denote the lines from $\mathcal{L}'_\tau$, intersecting $C_i$.
Clearly, $|\mathcal{P}' \cap C_i| = |\G| |\Pi|$.
Further, because of any line $l$ meets at most $|\Omega_\D (l)| \ll |\G| |\Pi| \D^{-3}$ cells $C_i$, we see that
$$
    \sum_{i=1}^s |L_i| \ll |\mathcal{L}'_\tau| \cdot |\G| |\Pi| \D^{-3} \,.
$$
Using the estimate $|\Omega_\D (l)| \D \gg \tau L^{-1}$,
the Cauchy--Schwarz inequality and the lower bound for $I(\mathcal{L}'_\tau, \mathcal{P}')$,
we obtain
$$
    |\mathcal{L}_\tau| L^{-1} \ll |\mathcal{L}'_\tau|
        \ll \frac{L |\G|^2 |\Pi|^2 s}{\D \tau} \ll \frac{L |\G|^2 |\Pi|^2 s}{\tau \max\{ 1, \tau s^{-1}\}}\,.
$$
After some calculations we have (\ref{f:E_Q1Q2}).
This completes the proof.
$\hfill\Box$
\end{proof}

\begin{corollary}
    Let $p$ be a prime number, $\Gamma$ be a multiplicative subgroup of $\mathbb{F}_p^*$, $|\G| < \sqrt{p}$,
    and $Q$ be $\G$--invariant set.
    Then there is $q\in Q$ such that
$$
    \E^{\times} (Q-q,\G + x) \ll |Q|^2 \log^2 |\G| \,,
$$
    where $x\in \f_p^*$ is an arbitrary.
\end{corollary}

\begin{remark}
    Considering $\T(Q_1,Q_2,\xi Q_1, \eta Q_2)$, where $\xi \neq 0,1$ or $\eta \neq 0,1$ one can reduce
    the term $|Q_1| |Q_2| (|Q_1|^2 + |Q_2|^2)$ in formula (\ref{f:E_Q1Q2}) of Theorem \ref{t:E_Q1Q2}
    sometimes.
    For example, if $\G$ is a subgroup, $Q$ is $\G$--invariant set then the correspondent error term in
    $\T(\G,Q,\xi \G, Q)$, $\xi \neq 0,1$ is $O(|\G|^3 |Q| + |\G|^2 |Q|^2)$, thus it is negligible.
\end{remark}

\bigskip

\noindent{I.D.~Shkredov\\
Steklov Mathematical Institute,\\
ul. Gubkina, 8, Moscow, Russia, 119991}
\\
and
\\
IITP RAS,  \\
Bolshoy Karetny per. 19, Moscow, Russia, 127994\\
{\tt ilya.shkredov@gmail.com}

\end{document}